\documentclass[11pt]{article}   	
\usepackage{geometry}                		
\usepackage{graphicx}				
\usepackage{amssymb}

\usepackage{enumerate}
\usepackage{enumitem}
\setenumerate{label={\normalfont (\alph*)}}

\usepackage{amsmath}
\usepackage{amssymb}
\usepackage{amsthm}

\usepackage{url}

\newtheorem{theorem}{Theorem}[section]
\newtheorem{lemma}[theorem]{Lemma}
\newtheorem{corollary}[theorem]{Corollary}
\newtheorem{result}[theorem]{Result}

\begin{document}
\title{On the Mean Connected Induced Subgraph Order of Cographs }
\author{Matthew E. Kroeker\thanks{Supported by an NSERC USRA}, Lucas Mol, and Ortrud R. Oellermann\thanks{Supported by an NSERC Grant CANADA, Grant number RGPIN-2016-05237}\\ University of Winnipeg, MB R3B 2E9\\
\small \url{mattkroeker@shaw.ca, l.mol@uwinnipeg.ca,
 o.oellermann@uwinnipeg.ca}}


\maketitle
\begin{abstract}
In this article the extremal structures for the mean order of connected induced subgraphs of cographs are determined. It is shown that among all connected cographs of order $n \ge 7$, the star $K_{1,n-1}$ has maximum mean connected induced subgraph order, and for $n \ge 3$, the $n$-{\em skillet}, $K_1+(K_1 \cup K_{n-2})$, has minimum mean connected induced subgraph order. It is deduced that the {\em density} for connected cographs (i.e.\ the ratio of the mean to the order of the graph) is asymptotically $1/2$. The  mean order of all connected induced subgraphs containing a given vertex $v$ of a cograph $G$, called the {\em local mean} of $G$ at $v$, is shown to be at least as large as the mean order of all connected induced subgraphs of $G$, called the {\em global mean} of $G$.
\end{abstract}

\section{Introduction}

The study of the mean order of subtrees of a tree was initiated by Jamison in \cite{j1, j2}. Several open problems posed in \cite{j1} were subsequently studied in \cite{mo, vw, ww1,ww}. In this article we extend the concept of the mean subtree order of trees to other graphs. Since the subtrees of a tree are precisely the connected induced subgraphs of a tree it is natural to study the problem of finding the mean order of the connected induced subgraphs of (connected) graphs.

If $G$ is a graph of order $n$ and $a_k(G)$ denotes the number of connected induced subgraphs of order $k$ in $G$ for each $k\in\{1,\dots,n\}$, then the {\em connected induced subgraph polynomial} of $G$ is given by
\[
\Phi_G(x) =\sum_{k=1}^n a_k(G)x^k.
\]
The mean order of the connected induced subgraphs of $G$, called the {\em global mean} of $G$, is given by the logarithmic derivative
\[
M_G =\frac{\Phi_G'(1)}{\Phi_G(1)}.
\]
We note that the polynomial $\Phi_G(x)$ is related to the \emph{residual node connectedness reliability} (or \emph{node reliability}), first studied in \cite{NRComplexity}.  For $p \in (0,1)$ the node reliability of $G$ is given by
\[
R_G(p)=\sum_{k=1}^na_k(G)p^k(1-p)^{n-k}.
\]
It is readily seen that
\[
\Phi_G(x)=(1+x)^nR_G\left(\tfrac{x}{1+x}\right).
\]
Given that each vertex of $G$ operates independently with probability $p$, the node reliability $R_G(p)$ is the probability that the operational vertices induce a connected subgraph.  In particular, it was shown in \cite{NRComplexity} that the problem of counting the total number of connected induced subgraphs of $G$ (that is, finding $\Phi_G(1)$) is \#P-complete, even for split graphs and for the family of graphs that are both planar and bipartite.  It follows that computing either $R_G(p)$ or $\Phi_G(x)$ is NP-hard in general.  However, polynomial algorithms have been found for certain restricted families \cite{NRPolynomial}, including trees, series-parallel graphs, and permutation graphs.  Several authors have considered optimality questions for node reliability \cite{NR1,NR2,NR3}, while others have investigated the (complex) roots of the polynomial \cite{NRRoots} and the shape on the interval $(0,1)$ \cite{NRShape}.

Let $v$ be a vertex of $G$, and for each $k\in\{1,\dots,n\}$ let $a_k(G;v)$ denote the number of connected induced subgraphs of $G$ containing $v$ and having order $k$.  Then the generating polynomial for the connected induced subgraphs of $G$ containing $v$, called the {\em local connected induced subgraph polynomial} of $G$ at $v$, is given by
\[
\Phi_{G,v}(x) =\sum_{k=1}^na_k(G;v)x^k.
\]
The mean order of the connected induced subgraphs containing $v$, called the {\em local mean} of $G$ at $v$, is given by the logarithmic derivative $M_{G,v} = \Phi'_{G}(v;1)/ \Phi_G(v;1)$. The {\em density} of a graph $G$ of order $n$ is defined as $M_G/n$. The density of a graph $G$ is the probability that a randomly chosen vertex belongs to a randomly chosen connected induced subgraph of $G$.

This paper focuses on  the mean order of connected induced subgraphs of `cographs'. A graph is a {\em cograph} if it contains no induced path on $4$ vertices.  Thus a graph is a cograph if and only if its complement is a cograph.  The following characterization of cographs proved in \cite{clb} will be used throughout the paper.  For graphs $G_1$ and $G_2$, the union of $G_1$ and $G_2$ is denoted by $G_1\cup G_2$ while the join is denoted by $G_1+G_2.$  The complement of a graph $G$ is denoted by $\overline{G}.$

\begin{result}\label{cographcharact}
Let $G$ be a graph.  The following are equivalent:
\begin{enumerate}
\item $G$ is a cograph.
\item $G \cong K_1$, or there exist two cographs $G_1$ and $G_2$ such that either $G=G_1 \cup G_2$, or $G=G_1+G_2$.
\item For each $S\subseteq V(G)$ with $|S|\geq 2$, $S$ induces a connected subgraph of $G$ if and only if $S$ induces a disconnected subgraph of $\overline{G}$ (i.e.\ each nontrivial vertex set induces a connected graph in exactly one of $G$ or $\overline{G}$).
\end{enumerate}
\end{result}

This result implies that $G$ is a connected cograph if and only if $G \cong K_1$ or $G$ is the join of two cographs. Moreover, $G$ is a nontrivial connected cograph if and only if its complement is a disconnected cograph.

We also make frequent use of the following straightforward result from \cite{j1} on logarithmic derivatives.  Essentially, if we express $\Phi_G(x)$ as a sum of polynomials with nonnegative coefficients, then $M_G$ is a weighted average of the logarithmic derivatives (means) of the polynomials in the sum.  Here and throughout this article, we use the convention that the logarithmic derivative of the zero polynomial is $0$.

\begin{result}\label{ConvexComboResult}
Let $\Phi_1,\dots,\Phi_k$ be polynomials with nonnegative coefficients.  Then the logarithmic derivative of $\Phi_1+\dots+\Phi_k$ is a convex combination of the logarithmic derivatives of the $\Phi_i$, i.e.\ $$\frac{\Phi_1'(1)+ \ldots + \Phi_k'(1)}{\Phi_1(1)+ \ldots + \Phi_k(1)} = \frac{\Phi_1'(1)}{\Phi_1(1)}c_1 + \ldots + \frac{\Phi_k'(1)}{\Phi_k(1)} c_k$$ where $c_1+ \ldots + c_k =1$, and $c_i\geq 0$ for all $i\in\{1,\dots,k\}$.
\end{result}

In several proofs throughout the article, we make use of a variant of the global mean. Let $G$ be a graph of order $n$, and let
\[
\Phi^*_G(x)=\Phi_G(x)-nx.
\]
The logarithmic derivative of $\Phi^*_G,$ denoted $M^*_G$, is the mean order of the nontrivial (that is, order at least $2$) connected induced subgraphs of $G$.  By our convention that the logarithmic derivative of a zero polynomial is $0$, if $G$ has no edges, then $M^*_G=0.$  The value $M^*_G$ is called the \emph{$M^*$ mean} of $G.$

\begin{figure}[h!]
\begin{center}
    \includegraphics{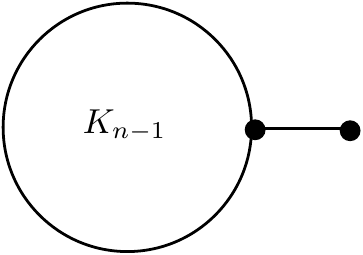}
    \caption{An $n$-skillet}
    \label{skillet}
\end{center}
\end{figure}

The layout of this article is as follows.  In Section \ref{MaxSection}, we show that among all connected cographs of order $n\ge 7$, the star $K_{1,n-1}$ has maximum mean connected induced subgraph order. In Section \ref{LocalSection}, we show that the local mean at any vertex of a connected cograph is at least as large as the global mean.  In Section \ref{MinSection}, we show that among all connected cographs of order $n \ge 3$ the $n$-{\em skillet}  $S_n=K_1+(K_1\cup K_{n-2})$ (see Fig.~\ref{skillet}), has minimum mean connected induced subgraph order. Finally, in Section \ref{DisconnectedSection}, we find the extremal structures for the mean connected induced subgraph order of disconnected cographs of order $n.$ We conclude by observing  that the density of connected cographs is asymptotically $1/2$.

\section{Connected cographs with maximum mean}\label{MaxSection}
We show in this section that among all connected cographs of order $n\geq 7$, the star $K_{1,n-1}$ has maximum global mean.  First note that
\[
\Phi_{K_{1,n-1}}(x)=x(1+x)^{n-1}+(n-1)x,
\]
so it is easy to verify that
\[
M_{K_{1,n-1}}=\frac{(n+1)\cdot 2^{n-2}+n-1}{2^{n-1}+n-1}=\frac{n+1}{2}-\frac{(n-1)^{2}}{2(2^{n-1}+n-1)}.
\]
The fact that this mean is maximum among all connected cographs of order $n\geq 7$ follows from two intermediate results which we prove in this section:
\begin{enumerate}
\item Among all complete bipartite graphs $K_{s, n-s}$, where $1\leq s\leq n-1$, the star $K_{1,n-1}$ has largest mean whenever $n\geq 7$.
\item If $G= G_{1} + G_{2}$, where $|V(G_{1})|=s$ and $|V(G_{2})|=n-s$, then $M_G \leq M_{K_{s,n-s}}$.
\end{enumerate}

We begin with some useful observations.  Let $G= G_{1} + G_{2}$, where $G_{1}$ and $G_{2}$ are graphs of orders $s$ and $n-s$, respectively. Then the connected induced subgraph polynomial of $G$ is given by
\[
\Phi_G (x) = \Phi_{G_{1}} (x) + \Phi_{G_{2}} (x) + \Phi^*_{K_{s,n-s}}(x),
\]
as the connected induced subgraphs of $G$ lie completely in either $G_1$ or $G_2$, or have at least one vertex from each graph (and every vertex of $G_1$ is joined to every vertex of $G_2$).  The polynomial $\Phi^*_{K_{s,n-s}}(x)$ plays an important role in much of our work.  By a basic counting argument, we have
\begin{align*}
\Phi^*_{K_{s,n-s}} (x) &=((1+x)^{s} - 1)((1+x)^{n-s} - 1)\\
&=(1+x)^n-(1+x)^s-(1+x)^{n-s}+1.
\end{align*}
For ease of notation, we let $\psi_{s,n-s}(x)=\Phi^*_{K_{s,n-s}}(x)$, so that
\begin{align}
\psi_{s,n-s}(1)&=2^n-2^{n-s}-2^s+1,\label{psi1}\\
\psi'_{s,n-s}(1)&=n2^{n-1}-s2^{s-1}-(n-s)2^{n-s-1},\label{psi2}
\end{align}
and
\begin{align*}
M^{\ast}_{K_{s,n-s}}=\frac{\psi'_{s,n-s}(1)}{\psi_{s,n-s}(1)}=\frac{n2^{n-1}-s2^{s-1}-(n-s)2^{n-s-1}}{2^n-2^{n-s}-2^s+1}.
\end{align*}

In order to show that $K_{1,n-1}$ has largest global mean among all complete bipartite graphs of order $n\geq 7$, we first show that $K_{1,n-1}$ has largest $M^*$ mean among all complete bipartite graphs of order $n\geq 1$.  The result is trivial for $n\leq 3$, and for $n\geq 4$ we prove the stronger result that $M^*_{K_{s,n-s}}$ is decreasing in $s$ for $s\in\left\{1,\dots,\lfloor\tfrac{n}{2}\rfloor\right\}.$  This stronger result is used to prove several of our later results, so we include it in the Lemma statement below.

\begin{lemma}\label{DecreaseLemma}
For $n\geq 4$ and $s\in\left\{1,\dots,\left\lfloor\tfrac{n}{2}\right\rfloor-1\right\},$
\[
M^*_{K_{s,n-s}}>M^*_{K_{s+1,n-s-1}}.
\]
In particular, $K_{1,n-1}$ has maximum $M^*$ mean among all complete bipartite graphs of order $n$ for all $n\geq 1$.
\end{lemma}

\begin{proof} First note that $s\leq\left\lfloor\tfrac{n}{2}\right\rfloor-1$ is equivalent to $n\geq 2s+2$.  We use this inequality throughout the proof. We have
\begin{align}
M^*_{K_{s,n-s}}-M^*_{K_{s+1,n-s-1}}&=\frac{\psi'_{s,n-s}(1)}{\psi_{s,n-s}(1)}-\frac{\psi'_{s+1,n-s-1}(1)}{\psi_{s+1,n-s-1}(1)}\nonumber\\
&=\frac{\psi'_{s,n-s}(1)\psi_{s+1,n-s-1}(1) - \psi'_{s+1,n-s-1}(1)\psi_{s,n-s}(1)}{\psi_{s,n-s}(1)\psi_{s+1,n-s-1}(1)}.\label{MeanKs}
\end{align}

It suffices to show that the numerator of the expression in (\ref{MeanKs}) is positive for $1\leq s\leq \left\lfloor\tfrac{n}{2}\right\rfloor-1.$
Substituting the appropriate expressions from (\ref{psi1}) and (\ref{psi2}) into the numerator of (\ref{MeanKs}), and then simplifying yields
\begin{align}\label{maineqn}
\begin{split}
&2^{2n-2s-2}-(n-s+1)\cdot 2^{n-s-2}+ (s-1)\cdot 2^{2n-s-2}-(n-s-2)\cdot 2^{n+s-1}\\
& \ \ \ \ +3(n-2s-1)\cdot 2^{n-2} - 2^{2s} + (s+2)\cdot 2^{s-1}.
\end{split}
\end{align}
The proof that (\ref{maineqn}) is strictly positive requires two cases:
\begin{flushleft}
{\bf Case 1: $s=1$.} Evaluating (\ref{maineqn}) at $s=1$ and simplifying gives
\end{flushleft}
\[
2^{2n-4}-(3n-6)2^{n-3} - 1,
\]
which is easily verified to be positive for $n\geq 4$ by induction.

\noindent{\bf Case 2: $2 \leq s \leq \left\lfloor \frac{n}{2} \right\rfloor - 1$.}
Since $n\geq 2s+2,$ we may assume that $n\geq 6.$  Grouping the terms of (\ref{maineqn}), we claim that
\begin{align}
2^{2n-2s-2} - (n-s+1)\cdot 2^{n-s-2} &> 0,\label{FirstGroup}\\
(s-1)\cdot 2^{2n-s-2}-(n-s-2)\cdot 2^{n+s-1}&\geq 0, \mbox{ and} \label{SecondGroup} \\
3(n-2s-1)\cdot 2^{n-2} - 2^{2s}&\geq 2^{n-1}, \label{ThirdGroup}
\end{align}
which together imply that (\ref{maineqn}) is strictly positive.  While we only need to show that $3(n-2s-1)\cdot 2^{n-2}-2^{2s}$ is nonnegative here, the stronger inequality (\ref{ThirdGroup}) is used in the proof of our next result.

For (\ref{FirstGroup}), we factor the left-hand side to obtain
\[
2^{n-s-2}\left[2^{n-s}-(n-s+1)\right],
\]
so it suffices to show that $2^{n-s}-(n-s+1)>0.$  Let $k=n-s$, and note that $k\geq s+2\geq 4.$  The inequality $2^k-(k+1)>0$ is easily proven by induction, completing the proof of (\ref{FirstGroup}).

For (\ref{SecondGroup}), factoring the left-hand side gives
\[
2^{n+s-1}\left[(s-1)\cdot 2^{n-2s-1}-(n-s-2)\right],
\]
so it suffices to show that $(s-1)\cdot 2^{n-2s-1}-(n-s-2)\geq 0.$  For each fixed value of $s\geq 2$, we prove this inequality by induction on $n.$  For the base case, we substitute $n=2s+2,$ and we find
\[
(s-1)\cdot 2-s=s-2\geq 0.
\]
Now suppose that the inequality holds for some natural number $n\geq 2s+2$.  Then
\begin{align*}
(s-1)\cdot 2^{n-2s}-(n-s-1)&=(s-1)\cdot 2^{n-2s-1}-(n-s-2)+(s-1)\cdot 2^{n-2s-1}-1\\
&\geq (s-1)\cdot 2^{n-2s-1}-1\\
&>0,
\end{align*}
completing the proof of (\ref{SecondGroup}).

Finally, for (\ref{ThirdGroup}), we use the inequality $n\geq 2s+2$ twice:
\[
3(n-2s-1)\cdot 2^{n-2} - 2^{2s}\geq 3\cdot 2^{n-2}-2^{2s}\geq 3\cdot 2^{n-2}-2^{n-2}=2^{n-1}.\qedhere
\]
\end{proof}

We now prove the first of the two statements from which our main result will follow.

\begin{theorem}\label{FirstMainLemma}
For $n\geq 7$,
\[
M_{K_{1,n-1}}>M_{K_{2,n-2}},
\]
and for $n\geq 6$ and $s\in\left\{2,\dots,\left\lfloor\tfrac{n}{2}\right\rfloor-1\right\},$
\[
M_{K_{s,n-s}}>M_{K_{s+1,n-s-1}}.
\]
In particular, among all complete bipartite graphs of order $n\geq 7,$ the star $K_{1,n-1}$ has maximum global mean.
\end{theorem}
\begin{proof} For any $n$ and any $s\in\{1,\dots, n-1\},$ we have
\[
\Phi_{K_{s,n-s}}=nx+\psi_{s,n-s}(x),
\]
so that
\begin{align*}
     M_{K_{s,n-s}}= \frac{n+\psi_{s,n-s}^{'}(1)}{n+\psi_{s,n-s}(1)}
\end{align*}
Hence, for any $s\in\left\{1,\dots,\left\lfloor\tfrac{n}{2}\right\rfloor-1\right\}$,
\begin{align*}
     &M_{K_{s,n-s}}-M_{K_{s+1,n-s-1}}\\
     & \ \ \ \ =\frac{[n+\psi_{s,n-s}^{'}(1)][n+\psi_{s+1,n-s-1}(1)] - [n+\psi_{s,n-s}(1)][n+\psi_{s+1,n-s-1}^{'}(1)]}{[n+\psi_{s,n-s}(1)]
     [n+\psi_{s+1,n-s-1}(1)]}.
\end{align*}
Expanding the numerator gives
\begin{align}\label{equation}
\begin{split}
&\psi_{s,n-s}^{'}(1)\psi_{s+1,n-s-1}(1) - \psi_{s,n-s}(1)\psi_{s+1,n-s-1}^{'}(1)\\
& \ \ \ \ +n\left[\psi_{s,n-s}^{'}(1) + \psi_{s+1,n-s-1}(1) - \psi_{s,n-s}(1) - \psi_{s+1,n-s-1}^{'}(1)\right],
\end{split}
\end{align}
and it suffices to show that this quantity is positive for the given values of $n$ and $s.$

We first demonstrate that (\ref{equation}) is positive when $s=1$ and $n\geq 7$, proving the first part of the Theorem.  When $s=1$, using (\ref{psi1}) and (\ref{psi2}), we find that (\ref{equation}) simplifies to
\[
2^{2n-4}-(n^2+n-6)\cdot 2^{n-3}+n-1,
\]
which is easily seen to be positive for $n\geq 7$ by induction.

Now let $n\geq 6$ and $s\in\left\{2,\dots,\left\lfloor\tfrac{n}{2}\right\rfloor-1\right\}$.  An expression for the first line of (\ref{equation}) is given by (\ref{maineqn}).   Applying inequalities (\ref{SecondGroup}) and (\ref{ThirdGroup}), proven in Lemma \ref{DecreaseLemma}, we obtain 
\begin{align*}
&\psi_{s,n-s}^{'}(1)\psi_{s+1,n-s-1}(1) - \psi_{s,n-s}(1)\psi_{s+1,n-s-1}^{'}(1)\\
&\ \ \ \ \geq 2^{2n-2s-2}-(n-s+1)\cdot 2^{n-s-2}+2^{n-1}.
\end{align*}
For the expression on the second line of (\ref{equation}), employing (\ref{psi1}) and (\ref{psi2}) and then simplifying, we find
\begin{align*}
     \psi_{s,n-s}^{'}(1) + \psi_{s+1, n-s-1}(1) - \psi_{s,n-s}(1) - \psi_{s+1, n-s-1}^{'}(1)&=s\cdot 2^{s-1}-(n-s-1)\cdot 2^{n-s-2}\\
     &>-(n-s-1) \cdot 2^{n-s-2}
\end{align*}
It follows that (\ref{equation}) is greater than
\begin{align*}
&2^{2n-2s-2}-(n-s+1)\cdot 2^{n-s-2}+2^{n-1}-n(n-s-1)\cdot 2^{n-s-2}\\
&\ \ \ \ =2^{n-s-2}\left[2^{n-s}-(n-s+1)+2^{s+1}-n(n-s-1)\right]\\
& \ \ \ \ =2^{n-s-2}\left[2^{n-s}+2^{s+1}-(n^2-ns-s+1)\right],
\end{align*}
so it suffices to show that
\begin{align}
2^{n-s}+2^{s+1}-(n^2-ns-s+1)>0. \label{simpeqn}
\end{align}
For each fixed value of $s\geq 2$, we proceed by induction on $n$.  Since $n\geq 2s+2$, we begin with the base case $n=2s+2.$  In this case, the left-hand side of (\ref{simpeqn}) evaluates to
\[
2^{s+2}+2^{s+1}-(2s^2+5s+5)=3\cdot 2^{s+1}-(2s^2+5s+5),
\]
which is positive for $s\geq 2$ by induction.  Now suppose that (\ref{simpeqn}) holds for some $n\geq 2s+2$.  Then
\begin{align*}
&2^{n+1-s}+2^{s+1}-\left[(n+1)^2-(n+1)s-s+1\right]\\
&=2^{n-s}+[2^{n-s}+2^{s+1}-(n^2-ns-s+1)]-(2n+1)+s\\
&> 2^{n-s}-(2n+1)+s,
\end{align*}
so it suffices to show that $2^{n-s}-(2n+1)+s\geq 0.$  We omit the details, but for each fixed $s\geq 2$, the inequality $2^{n-s}-(2n+1)>0$ can be shown to hold for $n\geq 2s+2$ by induction.  This completes the proof.
\end{proof}

To prove the second of our main two statements the following three lemmas on $M^{\ast}$ means will be useful.

\begin{lemma}\label{IncreaseLemma}
$M^{\ast}_{K_{1,n-1}}$ is strictly increasing for $n \geq 1$. Moreover, $\tfrac{n+1}{2}<M^{\ast}_{K_{1,n-1}}\leq \tfrac{n+2}{2}$ for all $n\geq 2.$
\end{lemma}
\begin{proof}
First of all, note that $M^*_{K_{1,1}}=2>0=M^*_{K_{1,0}}$.  Now for $n\geq 2$, we have
\[
M^{\ast}_{K_{1,n-1}}=\frac{2^{n-2}(n+1)-1}{2^{n-1}-1}=\frac{n+1}{2}+\frac{n-1}{2(2^{n-1}-1)},
\]
so it suffices to show that $0< \tfrac{n-1}{2(2^{n-1}-1)}\leq \tfrac{1}{2}$ for $n\geq 2.$  The first inequality is obvious, while the second holds by a straightforward induction.
\end{proof}

\begin{lemma}\label{LogDeriv}
Let G be a cograph (not necessarily connected) of order $n \geq 1$. Then $M^{\ast}_{G} \le M^{\ast}_{K_{1,n-1}}$, with equality if and only if $G\cong K_{1,n-1}.$
\end{lemma}

\begin{proof}
We proceed by induction on $n$. The statement is trivial for $n=1$. Suppose now that $n\geq 2$ and that the statement holds for every integer $k$, $1 \le k  < n$. Since $G$ is a cograph, it is either the join of two cographs or the union of two cographs.

\noindent
{\bf Case 1:} $G=G_{1}+G_{2}$, where $G_{1}$ and $G_{2}$ are cographs of orders $s$ and $n-s$ respectively.

In this case,
\begin{align*}
     \Phi_{G}(x) = \left[\Phi_{G_{1}}(x) - sx\right] + \left[\Phi_{G_{2}}(x) - (n-s)x\right]
     + \Phi_{K_{s,n-s}}(x),
\end{align*}
and hence
\begin{align*}
     \Phi_{G}(x) -nx = \left[\Phi_{G_{1}}(x) - sx\right] + \left[\Phi_{G_{2}}(x) - (n-s)x\right]
     + \left[\Phi_{K_{s,n-s}}(x) - nx\right].
\end{align*}
Thus, by taking the logarithmic derivative of both sides, we obtain the following expression:
\begin{equation}\label{convexcom}
     M^{\ast}_{G} = c_{1}M^{\ast}_{G_{1}} + c_{2}M^{\ast}_{G_{2}} + c_{3}M^{\ast}_{K_{s,n-s}},
\end{equation}
for real numbers $c_1,c_2,c_3 \ge 0$ such that $c_{1}+c_{2}+c_{3}=1$. By the induction hypothesis and then Lemma \ref{IncreaseLemma},
\begin{align*}
     M^{\ast}_{G_{1}} \leq M^{\ast}_{K_{1,s-1}}< M^{\ast}_{K_{1,n-1}} \ \ \ \mbox{and} \ \ \
     M^{\ast}_{G_{2}} \leq M^{\ast}_{K_{1,n-s-1}}< M^{\ast}_{K_{1,n-1}},
\end{align*}
and by Lemma \ref{DecreaseLemma},
\begin{align*}
     M^{\ast}_{K_{s,n-s}} \le M^{\ast}_{K_{1,n-1}}.
\end{align*}
Therefore, since each of the means on the right-hand side of (\ref{convexcom}) is at most $M^{\ast}_{K_{1,n-1}},$ we must have
\begin{align*}
M^{\ast}_{G} \leq M^{\ast}_{K_{1,n-1}}.
\end{align*}
Since $M^*_{G_1}<M^*_{K_{1,n-1}}$ and $M^*_{G_2}<M^*_{K_{1,n-1}}$, we have equality if and only if $c_1=c_2=0$, or equivalently, $G\cong K_{1,n-1}.$

\noindent
{\bf Case 2:} $G=G_{1} \cup G_{2}$, where $G_{1}$ and $G_{2}$ are cographs of orders $s$ and $n-s$ respectively.

In this case,
\begin{align*}
     \Phi_{G}(x) -nx = \left[\Phi_{G_{1}}(x) - sx\right] + \left[\Phi_{G_{2}}(x) - (n-s)x\right].
\end{align*}
Hence
\begin{align*}
     M^{\ast}_{G} = c_{1}M^{\ast}_{G_{1}} + c_{2}M^{\ast}_{G_{2}},
\end{align*}
for real numbers $c_1, c_2 \ge 0$ such that $c_{1}+c_{2}=1$. By the induction hypothesis and Lemma \ref{IncreaseLemma},
\begin{align*}
     M^{\ast}_{G} &\leq c_{1}M^{\ast}_{K_{1,s-1}} + c_{2}M^{\ast}_{K_{1,n-s-1}}< c_{1}M^{\ast}_{K_{1,n-1}} + c_{2}M^{\ast}_{K_{1,n-1}}=M^*_{K_{1,n-1}}.\qedhere
\end{align*}
\end{proof}

\begin{lemma}\label{LogDerivandMean}
Let $n\geq 6$ and $s\in\{1,\dots,n-1\}.$
\begin{enumerate}
\item If $G$ is a cograph on $k\leq n-2$ vertices, then $M^*_{G}<M_{K_{s,n-s}}.$
\item If $G$ is a cograph on $n-1$ vertices, then $M^*_{G}<M_{K_{1,n-1}}.$
\end{enumerate}

\end{lemma}
\begin{proof} Part (a): Without loss of generality we may assume that $s\leq \left\lfloor \tfrac{n}{2}\right\rfloor$, or equivalently $n\geq 2s$.  By Lemma \ref{LogDeriv},
$M^*_G\leq M^*_{K_{1,k-1}},$
so it suffices to show that $M^*_{K_{1,k-1}}<M_{K_{s,n-s}}.$  From  Lemma \ref{IncreaseLemma}, we have
\[
M^*_{K_{1,k-1}}\leq \tfrac{k+2}{2}\leq \tfrac{n}{2},
\]
so it suffices to show that $M_{K_{s,n-s}}>\tfrac{n}{2}.$  Indeed, we have
\begin{align}
     M_{K_{s,n-s}} - \frac{n}{2}&= \frac{n2^{n-1}-s2^{s-1}-(n-s)2^{n-s-1}+n}{2^{n}-2^{s}-2^{n-s}+1+n}
     - \frac{n}{2} \nonumber\\
     &= \frac{(n-s)2^{s}+s2^{n-s}+n-n^2}{2\left[2^{n}-2^{s}-2^{n-s}+1+n \right]} \label{MeanKs_n-s}
\end{align}
Using an argument similar to the one used to show that the expression in (\ref{simpeqn}) is positive, it can be argued for fixed $s \ge 1$ and by induction on $n \ge 2s$ that the numerator in (\ref{MeanKs_n-s}) is positive.

\noindent
Part (b): By Lemma \ref{LogDeriv},
$M^*_G\leq M^*_{K_{1,n-2}},$
so it suffices to show that $M^*_{K_{1,n-2}}<M_{K_{1,n-1}}.$  We have
\begin{align*}
     M_{K_{1,n-1}} - M^{\ast}_{K_{1,n-2}}= \frac{(n+1)2^{n-2}+(n-1)}{2^{n-1}+(n-1)} - \frac{n2^{n-3}-1}{2^{n-2}-1}.
\end{align*}
The numerator of this expression over common denominator $\left[2^{n-1}+(n-1)\right]\left[2^{n-2}-1\right]$ is given by
\begin{align*}
     &\ \ \ (n+1)2^{2n-4}+(n-1)2^{n-2}-(n+1)2^{n-2}-(n-1)\\
     & \ \ \ \ \ \ \ - \left[n2^{2n-4}+n(n-1)2^{n-3}-2^{n-1}-(n-1)\right]
     \\
     &= 2^{2n-4} + (n-1)2^{n-2} + 2^{n-1} - (n+1)2^{n-2} - n(n-1)2^{n-3} \\
     &= 2^{n-3}\left[2^{n-1}-n(n-1)\right],
\end{align*}
and $2^{n-1} > n(n-1)$ holds for $n \ge 6$ by induction.
\end{proof}

We now prove the second of the two main intermediate results of this section.

\begin{theorem}\label{SecondMainLemma}
Let $G$ be a connected cograph of order $n\geq 6$, obtained from the join of two cographs $G_{1}$ and $G_{2}$ of orders $s$ and $n-s$, respectively. Then $M_{G} \leq M_{K_{s,n-s}}$, with equality if and only if $G\cong K_{s,n-s}$.
\end{theorem}
\begin{proof}
Since $G=G_{1}+G_{2}$, the generating polynomial for the connected induced subgraphs of $G$ is given by
\begin{align*}
     \Phi_{G}(x) = \left[\Phi_{G_{1}}(x) -sx\right] + \left[\Phi_{G_{2}}(x) - (n-s)x \right] + \Phi_{K_{s,n-s}}(x).
\end{align*}
Thus,
\begin{align*}
     M_{G} = c_{1}M^{\ast}_{G_{1}} + c_{2}M^{\ast}_{G_{2}} + c_{3}M_{K_{s,n-s}}
\end{align*}
for some real numbers $c_1,c_2,c_3 \ge 0$ such that $c_1+c_2+c_3 =1$. By Lemma \ref{LogDerivandMean}, $M^{\ast}_{G_{1}}< M_{K_{s,n-s}}$ and $M^{\ast}_{G_{2}} < M_{K_{s,n-s}}$.  Therefore, $M_G$ is a weighted average of means of which $M_{K_{s,n-s}}$ is the largest, from which we conclude that $M_G\leq M_{K_{s,n-s}}.$  The inequality is strict unless $c_1=c_2=0$, i.e.\ $G\cong K_{s,n-s}$.
\end{proof}

Our main result now follows easily.

\begin{corollary}\label{FirstMainTheorem}
If $G$ is a connected cograph of order $n\geq 7$, then
\begin{align*}
M_{G}\leq \frac{n+1}{2}-\frac{(n-1)^{2}}{2(2^{n-1}+n-1)},
\end{align*}
with equality if and only if $G\cong K_{1,n-1}.$ Generally, if $G$ is a connected cograph of order $n\geq 1$, then $M_G\leq \tfrac{n+1}{2},$ with equality if and only if $n=1$.
\end{corollary}

\begin{proof}
The first statement follows immediately from Theorems \ref{FirstMainLemma} and \ref{SecondMainLemma}.  The second statement follows easily from the first in the case that $n\geq 7$ and is easily checked exhaustively for $n\leq 6$ using a computer algebra system.  The cographs of maximum global mean for $n\leq 6$ and their means are given in Table \ref{MaxTable}.  Decimal expansions are provided for ease of verifying this statement.
\end{proof}

\begin{table}[h]
\begin{center}
\begin{tabular}{c c c c}
Order & Cograph & Global Mean & Decimal Expansion\\
$1$ & $K_1$ & $1$ & $1$\\
$2$ & $K_2$ & $4/3$ & $1.33\dots$\\
$3$ & $K_3$ & $12/7$ & $1.71\dots$\\
$4$ & $K_{2,2}$ & $28/13$ & $2.15\dots$\\
$5$ & $K_{2,3}$ & $69/26$ & $2.65\dots$\\
$6$ & $K_{2,4}$ & $54/17$ & $3.18\dots$
\end{tabular}
\caption{The cographs of maximum global mean among all connected cographs of order $n\leq 6$.\label{MaxTable}}
\end{center}
\end{table}

Finally, we observe that our upper bound on the mean connected induced subgraph order of connected cographs generalizes nicely to all cographs, as described below.

\begin{corollary}\label{DisconnectedUpperBound}
Let $G$ be a cograph in which every component has order at most $s$.  Then $M_G\leq \tfrac{s+1}{2}$ with equality if and only if $s=1.$
\end{corollary}

\begin{proof}
Suppose $G$ has components $G_1,\dots, G_k$.  Since $\Phi_G(x)=\sum_{i=1}^k\Phi_{G_i}(x)$, the mean $M_G$ is a convex combination of the $M_{G_i}$, so the result follows immediately from Corollary \ref{FirstMainTheorem}.
\end{proof}

\section{Comparing local and global means}\label{LocalSection}

It was shown in \cite{j1} that
\[
M_{T,v}\geq M_T
\]
for any tree $T$ and any vertex $v$ of $T$, with equality if and only if $T\cong K_1.$  It is natural to ask whether this property extends to the mean connected induced subgraph order of graphs in general.  We provide an infinite family of graphs with a vertex for which this inequality does not hold.

For each natural number $n$, let $\mathcal{H}_n$ be the family of graphs of order $n$ with mean greater than $\tfrac{n+2}{2}.$  From \cite{mo}, there is a tree (in fact, a caterpillar) $T_n$ of order $n$ such that $M_{T_n}\geq n-\lceil 2\log_2 n\rceil -1$ for all $n$.  It follows that $\mathcal{H}_n$ is nonempty for all $n\geq 25,$ and further, there is a tree in $\mathcal{H}_n$ for all $13\leq n\leq 24$ by computational work.  Define
\[
\mathcal{G}_n=\{H+K_1\colon \ H\in \mathcal{H}_{n-1}\},
\]
and note that $\mathcal{G}_n$ is nonempty for all $n\geq 14.$

Let $G\in \mathcal{G}_n$ for some $n\geq 14$.  In particular, let $G=H+v$, where $H\in \mathcal{H}_{n-1}.$  Then
\begin{align}
\Phi_{G}(x)=\Phi_{H}(x)+\Phi_{G}(v;x).\label{LocalGlobal}
\end{align}
Since $v$ is universal in $G$, we see that $\Phi_G(x)=x(1+x)^{n-1}$, and thus $M_{G,v}=\tfrac{n+1}{2}.$  Since $H\in\mathcal{H}_{n-1},$ we have $M_H>\tfrac{(n-1)+2}{2}=\tfrac{n+1}{2}=M_{G,v}.$  We conclude that $M_H>M_G>M_{G,v}$, since $M_G$ is a convex combination of $M_H$ and $M_{G,v}$ by (\ref{LocalGlobal}).  Therefore, for each $n\geq 14$, $\mathcal{G}_n$ is nonempty and every member of $\mathcal{G}_n$ has a vertex at which the local mean is less than the global mean.

Using results from the previous section, we now demonstrate that $M_{G,v}\geq M_G$ for any connected \textit{cograph} G, and any vertex $v\in V(G).$  Aside from being interesting in its own right, this result will be used in the next section when we determine the structure of connected cographs with minimum global mean.

\begin{theorem}\label{LocalMean}
Let $G$ be a connected cograph of order $n \ge 1$ and  $v$ any vertex of $G$. Then $$M_{G,v} \geq \tfrac{n+1}{2}.$$
\end{theorem}

\begin{proof}
We proceed by induction on $n$. The result holds for $n=1$. Suppose $n \geq 2$ and that the result holds for every connected cograph of order less than $n$, and let $G$ be a connected cograph of order $n$.  Since $G$ is a connected cograph, $G=G_{1}+G_{2}$, where $G_{1}$ is a cograph of order $s$ (say), and $G_{2}$ is a cograph of order $n-s$. Without loss of generality, suppose that $v \in
V(G_{1})$. More specifically, since $G_{1}$ is not necessarily connected, $v$ belongs to a component $H$ of $G_{1}$.  Let $k$ be the order of $H$ (note that $1\leq k\leq s$).

Now the connected induced subgraphs of $G$ containing $v$ can be partitioned into those containing no vertex of $G_2$, and those containing at least one vertex of $G_{2}$.  Hence
\begin{align*}
     \Phi_{G,v}(x) = \Phi_{G_{1},v}(x) + x\left[\sum_{i=1}^{n-s} \tbinom{n-s}{i}x^{i}\right]
     \left[\sum_{j=0}^{s-1}\tbinom{s-1}{j}x^{j}\right].
\end{align*}
Thus,
\begin{align*}
     \Phi_{G,v}(x)&= \Phi_{G_{1},v}(x) + x\left[(1+x)^{n-s}-1\right](1+x)^{s-1} \\
     &= \Phi_{G_{1},v}(x) + x\left[(1+x)^{n-1}-(1+x)^{s-1}\right],
\end{align*}
and
\begin{align*}
     \Phi_{G,v}^{'}(x)&= \Phi_{G_{1},v}^{'}(x) + \left[(1+x)^{n-1}-(1+x)^{s-1}\right]\\
     & \ \ \ \ + x\left[(n-1)(1+x)^{n-2} - (s-1)(1+x)^{s-2}\right].
\end{align*}
Therefore,
\begin{align*}
     M_{G,v} = \frac{\Phi_{G,v}^{'}(1)}{\Phi_{G,v}(1)} =
     \frac{\Phi_{G_1,v}^{'}(1) + 2^{n-2}(n+1) - 2^{s-2}(s+1)}{\Phi_{G_{1},v}(1) + 2^{n-1} -
     2^{s-1}}.
\end{align*}
Straightforward algebra gives
\begin{align*}
     M_{G,v}-\tfrac{n+1}{2}=\frac{2^{s-2}(n-s) + \Phi_{G_{1},v}(1)\left[M_{G_{1},v} -
     \left(\frac{n+1}{2}\right)\right]}{\Phi_{G_{1},v}(1) + 2^{n-1} - 2^{s-1}},
\end{align*}
so it suffices to show that
\begin{align*}
    2^{s-2}(n-s) + \Phi_{G_{1},v}(1)\left[M_{G_{1},v} - \tfrac{n+1}{2}\right] \geq 0.
\end{align*}
By the induction hypothesis,
\begin{align*}
2^{s-2}(n-s) + \Phi_{G_{1},v}(1) \left[ M_{G_{1},v} - \tfrac{n+1}{2} \right]
     &\geq 2^{s-2}(n-s) + \Phi_{G_{1},v}(1) \left[\tfrac{k+1}{2} - \tfrac{n+1}{2} \right] \\
     &= 2^{s-2}(n-s) - \Phi_{G_{1},v}(1) \left( \tfrac{n-k}{2} \right).
\end{align*}
Moreover, $\Phi_{G_{1},v}(1)\leq 2^{k-1}$ with equality if and only if $v$ is a universal vertex of $H$, so we have
\begin{align*}
     2^{s-2}(n-s) - \Phi_{G_{1},v}(1)\left(\frac{n-k}{2}\right)&\geq 2^{s-2}(n-s) -
     2^{k-1}\left(\frac{n-k}{2}\right) \\
     &= 2^{s-2} (n-s)- 2^{k-2}(n-k) \\
     &= 2^{k-2}\left[(n-s)2^{s-k} - n + k\right].
\end{align*}

Now it suffices to show that $(n-s)2^{s-k}-n+k\geq 0.$  Regrouping and applying the inequality $n-s\geq 1$, we obtain
\[
(n-s) 2^{s-k}-n+k=(n-s)\left(2^{s-k}-1\right)+n-s-n+k\geq 2^{s-k}-1-s+k=2^{s-k}-(s-k)-1.
\]
The result follows by letting $m=s-k$, and observing that, by induction on $m \ge 0$, we have $2^m-m-1\geq 0$.
\end{proof}

The fact that the local mean at any vertex of a connected cograph is greater than the global mean now follows directly from Corollary \ref{FirstMainTheorem} and Theorem \ref{LocalMean}.

\begin{corollary}
If $G$ is a connected cograph with vertex $v$, then $M_{G,v}\geq M_G$, with equality if and only if $G\cong K_1.$
\end{corollary}

\begin{proof}
Let $G$ be a connected cograph of order $n$, and $v\in V(G)$.  By Corollary \ref{FirstMainTheorem} and then Theorem \ref{LocalMean}, we have
\[
M_{G}\leq \tfrac{n+1}{2}\leq M_{G,v},
\]
with equality at the first inequality if and only if $n=1$.
\end{proof}

\section{Connected cographs with minimum mean}\label{MinSection}

In this section we determine the structure of those connected cographs for which the global mean is a minimum. Recall that the $n$-skillet, denoted by $S_{n}$, is the cograph defined by $S_{n}=K_{1}+(K_{1} \cup K_{n-2})$, pictured in Fig.~\ref{skillet}.  We first demonstrate by direct proof that if $G$ is a connected cograph of order $n$ with a cut vertex, then $M_G\geq M_{S_n}$.  Then we consider the general case where $G$ is any connected cograph of order $n$; the proof is by induction.

We begin by deriving expressions for $M_{S_n}$ and $M_{K_n}$ and showing that $M_{S_n}<M_{K_n}$ for all $n\geq 3$.

\begin{lemma}\label{specialcases}
For any $n \ge 3$,
\begin{enumerate}
\item $M_{S_n} = \frac{n}{2}+ \frac{1}{3 \cdot 2^{n-2}}$,
\item $M_{K_n}= \frac{n}{2} + \frac{n}{2^{n+1}-2}$, and
\item $M_{S_n}<M_{K_n}$.
\end{enumerate}
\end{lemma}
\begin{proof}
Parts (a) and (b) are easily obtained from the connected induced subgraph polynomials for $S_n$ and $K_n$, respectively:
\begin{align*}
\Phi_{S_n}(x) &= (1+x)^{n-1}-1 + x+x^2(1+x)^{n-3}\\
\Phi_{K_n}(x) &= (1+x)^{n}-1
\end{align*}

For part (c), we show that the difference $M_{K_n}-M_{S_n}$ is positive for $n\ge 3$.  We have
\begin{align*}
M_{K_n}-M_{S_n}=\frac{n}{2^{n+1}-2}-\frac{1}{3\cdot 2^{n-2}}=\frac{2^{n-2}(3n-8)+2}{3\cdot 2^{n-2}\cdot (2^{n+1}-2)}>0
\end{align*}
where the last inequality follows easily from the fact that $n\geq 3$.
\end{proof}

The following three lemmas on $M^{\ast}$ means of cographs will be used to prove the main result of this section.

\begin{lemma}\label{ShortLemma}
$M^{\ast}_{K_{s,n-s}} > \frac{n}{2}$ for any $n\geq 2$ and $1\leq s\leq n-1$.
\end{lemma}
\begin{proof}
We have
\begin{align*}
     M^{\ast}_{K_{s,n-s}}-\frac{n}{2} &= \frac{n2^{n-1}-s2^{s-1}-(n-s)2^{n-s-1}}{2^{n}-2^{s}-2^{n-s}+1}
     - \frac{n}{2} \\
     &= \frac{(n-s)2^{s}+s2^{n-s}-n}{2(2^{n}-2^{s}-2^{n-s}+1)} \\
     &\geq \frac{2(n-s)+2s-n}{2(2^{n}-2^{s}-2^{n-s}+1)} \\
     &>0.
\end{align*}
Hence $M^{\ast}_{K_{s,n-s}} > \frac{n}{2}$.
\end{proof}

\begin{lemma}\label{ComBipartiteLemma}
Let $G=G_{1}+G_{2}$, where $G_{1}$ and $G_{2}$ are cographs of orders $s$ and $n-s$ respectively, where $n\geq 4$ and $2 \leq s \leq n-2$. Then $M^{\ast}_{G} \le M^{\ast}_{K_{s,n-s}}$, with equality if and only if $ G \cong K_{s,n-s}$.
\end{lemma}
\begin{proof}
Since $G$ is a connected cograph, observe that
\begin{align*}
     \Phi_{G}(x) - nx = \left[\Phi_{G_{1}}(x) - sx\right] + \left[\Phi_{G_{2}}(x) - (n-s)x\right] + \left[\Phi_{K_{s,n-s}}(x) - nx\right].
\end{align*}
Hence,
\begin{align}\label{ConvexCombo}
     M^{\ast}_{G} = c_{1}M^{\ast}_{G_{1}} + c_{2}M^{\ast}_{G_{2}} + c_{3}M^{\ast}_{K_{s,n-s}},
\end{align}
for some nonnegative numbers $c_1, c_2, c_3$, such that $c_1+c_2+c_3 =1$.
By Lemmas \ref{LogDeriv} and  \ref{LogDerivandMean}, and Lemma \ref{IncreaseLemma},
\begin{align*}
M^{\ast}_{G_{1}} \leq M^{\ast}_{K_{1,s-1}}\leq \tfrac{s+2}{2}\leq \tfrac{n}{2}, \mbox{ and}\\
M^{\ast}_{G_{2}} \leq M^{\ast}_{K_{1,n-s-1}}\leq \tfrac{n-s+2}{2}\leq \tfrac{n}{2}.
\end{align*}
Finally, by Lemma \ref{ShortLemma}, $M^{\ast}_{K_{s,n-s}}>\tfrac{n}{2}$, and hence $M^{\ast}_{G} < M^{\ast}_{K_{s,n-s}}$ follows from (\ref{ConvexCombo}).  Moreover, we have equality if and only if $c_1=c_2=0$, that is, if and only if $G\cong K_{s,n-s}.$
\end{proof}

\begin{lemma}\label{CompleteLemma}
$M^{\ast}_{K_{2,n-3}} \leq M_{K_{n}}$ for $n \geq 6$.
\end{lemma}
\begin{proof}
Consider the difference
\begin{align*}
M_{K_{n}} - M^{\ast}_{K_{2,n-3}} &= \frac{n2^{n-1}}{2^{n}-1} - \frac{2^{n-4}(3n-1)-4}{3\cdot2^{n-3}-3}\\
&=\frac{n2^{n-1}[3\cdot2^{n-3}-3] - (2^{n}-1)[2^{n-4}(3n-1)-4]}{(2^n-1)(3\cdot 2^{n-3}-3)}.
\end{align*}
Expanding and regrouping the numerator, we obtain
\begin{align*}
M_{K_{n}} - M^{\ast}_{K_{2,n-3}}      &= \frac{2^{n-4}\left(2^{n} - 21n +63\right) - 4}{(2^n-1)(3\cdot 2^{n-3}-3)}.
\end{align*}
Thus it suffices to show that $2^{n} - 21n + 63 \geq 1$ for $n \geq 6$, and this follows easily by induction.
\end{proof}

We now prove the first special case of the main result of this section, namely the case where $G$ is a connected cograph with a cut vertex. This case is further divided into the subcases where $G$ does or does not have a leaf. Let $G$ be a connected cograph of order $n$ with a cut vertex.  Then it is straightforward to show that $G=K_1 + H$ where $H$ is a disconnected cograph.  We use this structure to show that $M_G>M_{S_n}$.

\begin{lemma}\label{CutVertex}
Let $G$ be a connected cograph of order $n$. If $G$ has a cut vertex, then $M_{G}\geq M_{S_{n}},$ with equality if and only if $G\cong S_n$.
\end{lemma}
\begin{proof}
Write $G=K_1+H$, where $H$ is a disconnected cograph of order $n-1$.  We consider two cases, depending on whether or not $H$ has an isolated vertex.

Suppose first that $H$ has no isolated vertex. In this case, we must have $n\geq 5$.  If $n=5$, then $H\cong K_2\cup K_2$, and $M_G>M_{S_5}$ can be verified directly.  Now assume that $n\geq 6$. We will prove the stronger result that $M_G\geq M_{K_n}$, from which the desired result follows by Lemma \ref{specialcases}(c).  We can write
\begin{align}\label{CompleteDecomposition}
    \Phi_{K_{n}}(x) = \Phi_{G}(x)+ \Phi_{\overline{G}}(x) - nx,
\end{align}
since every nonempty subset of vertices induces a connected subgraph of $K_n$, while, by Result \ref{cographcharact}(c),  every subset of vertices of size at least $2$ induces a connected subgraph in either $G$ or $\overline{G}$, but not in both.  The singleton subsets are counted by both $\Phi_G(x)$ and $\Phi_{\overline{G}}(x)$, which is why we subtract $nx.$  Since $\overline{G} = K_{1} \cup \overline{H}$, we substitute $\Phi_{\overline{G}}(x)=\Phi_{\overline{H}}(x)+x$ into (\ref{CompleteDecomposition}) to obtain
\begin{align*}
    \Phi_{K_{n}}(x) = \Phi_{G}(x)+\left[\Phi_{\overline{H}}(x) - (n-1)x\right].
\end{align*}
Therefore, we can write $M_{K_n}$ as a convex combination as follows:
\begin{align} \label{ConvexCombo2}
     M_{K_{n}} =c_{1}M_{G} + c_{2}M^{\ast}_{\overline{H}}.
\end{align}
Since $\overline{H}$ is a connected cograph of order $n-1$, and $H$ has no isolated vertex, we must have $\overline{H}=H_1+H_2$ with $|V(H_1)|\geq 2$ and $|V(H_2)|\geq 2$.  Let $|V(H_1)|=s$. By Lemma \ref{ComBipartiteLemma}, $M^{\ast}_{\overline{H}} \leq M^{\ast}_{K_{s,n-1-s}}$.  Further, by Lemma \ref{DecreaseLemma}, $M^{\ast}_{K_{s,n-1-s}} \leq M^{\ast}_{K_{2,n-3}}$.  Finally, by Lemma \ref{CompleteLemma}, $M^{\ast}_{K_{2,n-3}} \leq M_{K_{n}}$ for $n\geq 6$. So $M^{\ast}_{\overline{H}} \leq M_{K_{n}}$, and we conclude from (\ref{ConvexCombo2}) that $M_G\geq M_{K_{n}}$.

Suppose otherwise that $H$ has an isolated vertex.  Then $G=K_{1}+H$  contains a leaf.  So $G$ is a spanning subgraph of $S_{n}$. Label the vertices of $G$ and $S_n$ with the same labels such that $v$ denotes a leaf in each graph and $u$ denotes the cut vertex of each graph. The sets of vertices that induce a connected subgraph in $S_n$ can be partitioned into two sets, namely (i) those that also induce a connected subgraph of $G$ and (ii) those that do not induce a connected subgraph of $G$. The latter are precisely those sets of vertices that induce a disconnected subgraph of $H-v$ and thus the nontrivial sets of vertices that induce a connected subgraph in $\overline{H-v}$. Hence
\begin{align*}
     \Phi_{S_{n}}(x) = \Phi_{G}(x)+\left[\Phi_{\overline{H - v}}(x) - (n-2)x\right].
\end{align*}
Thus we obtain the following convex combination:
\begin{align}\label{ConvexCombo3}
     M_{S_{n}} =  c_{1}M_{G} + c_{2}M^{\ast}_{\overline{H -v}}
\end{align}
Since $\overline{H-v}$ is a cograph of order $n-2$, Lemma \ref{LogDeriv} and Lemma \ref{IncreaseLemma} give
\begin{align*}
     M^{\ast}_{\overline{H - v}} \leq M^{\ast}_{K_{1,n-3}}\leq \tfrac{n}{2}.
\end{align*}
Finally, note that $M_{S_{n}}=\tfrac{n}{2}+\tfrac{1}{3\cdot 2^{n-2}}>\tfrac{n}{2}$. Thus $M^{\ast}_{\overline{H -v}} < M_{S_{n}}$, and we conclude from (\ref{ConvexCombo3}) that $M_{S_{n}} < M_{G}$.
\end{proof}

We now focus on the case where $G$ is a connected cograph without a cut vertex, i.e.\ $G$ is a $2$-connected cograph. In this case, $\Phi_G(x)$ can be expressed as the sum of a local generating polynomial $\Phi_{G,v}(x)$ and the generating polynomial $\Phi_{G-v}(x)$ for the connected cograph $G-v$. We require one more lemma in order to prove the main result of this section.

\begin{lemma}\label{CountingLemma}
Let $G$ be a $2$-connected cograph of order $n \geq 4.$  There is some vertex $v\in V(G)$ such that $\Phi_{G-v}(1) < \Phi_{G,v}(1)$.
\end{lemma}
\begin{proof}
Since $G$ is $2$-connected, there exist non-trivial cographs $G_{1}$ and $G_{2}$ such that $G=G_{1}+G_{2}$. Let $v$ be a vertex of maximum degree in $G$.  We may assume that $v \in V(G_{1})$. If $v$ is a universal vertex, then the result follows, since, in this case, there are $2^{n-1}$ connected induced subgraphs that contain $v$ and at most $2^{n-1}-1$ connected induced subgraphs in $G - v$.

Suppose now that $v$ has $k \geq 1$ non-neighbours. Let ${\cal P}_{1}$ be the set of connected induced subgraphs of $G-v$ that contain a neighbour of $v$ and ${\cal P}_{2}$ the set of connected induced subgraphs of $G-v$ that do not contain a neighbour of $v$.  If $H$ belongs to ${\cal P}_{1}$, then the vertices of $H$ together with $v$ induce a connected subgraph that contains $v$. Let ${\cal P}_{1,v}$ be the family of such connected induced subgraphs. If $H$ is in ${\cal P}_{2}$, then the vertices of $H$ belong to $V(G_{1})-N[v]$. So ${\cal P}_{2}$ has at most $2^{k}-1$ elements. Let $u \in V(G_{2})$. Since $\deg(u) \leq \deg(v)$, $u$ has at least $k$ non-neighbours that necessarily belong to $G_{2}$. Let $S$ be the set of non-neighbours of $u$. Each non-empty subset $X$ of $S$ together with $u$ and $v$ yields a connected induced subgraph that contains $v$ and does not belong to ${\cal P}_{1}$, since $\langle X \cup \{u\} \rangle$ is disconnected. Let ${\cal P}_{2,v}$ be the family of connected induced subgraphs of $G$ containing $v$ that are constructed in this manner. Note that
$| {\cal P}_{i,v} |\geq | {\cal P}_{i} |$ for $i=1,2$, and ${\cal P}_{1,v}\cap {\cal P}_{2,v}=\emptyset,$ which gives $\Phi_{G-v}(1)\leq \Phi_{G,v}(1)$. Finally, the singleton set containing $v$ induces a connected subgraph of $G$ containing $v$ and is not contained in ${\cal P}_{1,v}\cup {\cal P}_{2,v}$, which gives the desired strict inequality.
\end{proof}

\begin{theorem}\label{SecondMainTheorem}
Let $G$ be a connected cograph of order $n \ge 3$.  Then $M_G\geq \tfrac{n}{2}+\tfrac{1}{3\cdot 2^{n-2}},$ with equality if and only if $G\cong S_n.$
\end{theorem}
\begin{proof}
We begin by deriving a useful expression for $M_{S_n}.$  Let $u$ be a vertex of degree $n-2$ in $S_n$. First note that
\begin{align*}
     \Phi_{S_{n}}(x) = \Phi_{S_{n}, u}(x) + \Phi_{S_{n-1}}(x).
\end{align*}
Taking the logarithmic derivative of both sides, we obtain
\begin{align*}
     M_{S_{n}} = \left[ \frac{\Phi_{S_n,u}(1)}{\Phi_{S_{n},u}(1)+\Phi_{S_{n-1}}(1)} \right] M_{S_{n},u}
     + \left[\frac{\Phi_{S_{n-1}}(1)}{\Phi_{S_{n},u}(1)+\Phi_{S_{n-1}}(1)} \right] M_{S_{n-1}}.
\end{align*}
One easily verifies that $\Phi_{S_n,u}(1) = \Phi_{S_{n-1}}(1) = 2^{n-2}+2^{n-3}$ and $M_{S_n,u}=\tfrac{n+1}{2},$ which gives
\begin{align}\label{convexcombination}
     M_{S_{n}}=\tfrac{1}{2}\cdot \tfrac{n+1}{2}+\tfrac{1}{2}M_{S_{n-1}}
\end{align}

We now proceed with the proof of the statement by induction on $n$. The result is easily verified for $n=3$ as $M_{K_3}>M_{S_3}$ and these are the only distinct connected cographs of order $3$. Suppose now that $n >3$ and that the statement holds for all connected cographs of order $k$, $3 \le k < n$. If $G$ has a cut vertex, then we are done, by Lemma \ref{CutVertex}.  So we may assume that $G$ is $2$-connected.

By Lemma \ref{CountingLemma} there is a vertex $v \in V(G)$ such that $\Phi_{G, v}(1) > \Phi_{G - v}(1)$. We express $M_{G}$ as follows:
\begin{align*}
     M_{G} = \left[ \frac{\Phi_{G, v}(1)}{\Phi_{G, v}(1)+\Phi_{G - v}(1)} \right] M_{G, v} +
     \left[  \frac{\Phi_{G-v}(1)}{\Phi_{G, v}(1)+\Phi_{G - v}(1)} \right] M_{G - v}.
\end{align*}
From our choice of $v$, it follows that the coefficient of $M_{G,v}$ is greater than $\tfrac{1}{2}$, while the coefficient of $M_{G-v}$ is less than $\tfrac{1}{2}$.  Since $M_{G,v}\geq \tfrac{n+1}{2}$, by Theorem \ref{LocalMean}, and $M_{G-v}<\tfrac{n}{2}$, by Corollary \ref{FirstMainTheorem}, we have $M_{G,v}>M_{G-v}$. Thus
\begin{align*}
     M_{G} &> \tfrac{1}{2}M_{G,v} + \tfrac{1}{2}M_{G - v}
     \geq \tfrac{1}{2} \cdot \tfrac{n+1}{2} + \tfrac{1}{2}M_{G - v}.
\end{align*}
Applying the induction hypothesis to $G-v$ gives
\begin{align*}
M_G >\tfrac{1}{2}\cdot\tfrac{n+1}{2}+\tfrac{1}{2}M_{S_{n-1}}=M_{S_n},
\end{align*}
where the last equality is due to (\ref{convexcombination}).
\end{proof}

It follows easily from our results that the global means of connected cographs increase with order.

\begin{corollary}\label{ComparingOrder}
Let $G_1$ be a connected cograph of order $n_1\geq 2$, and let $G_2$ be a cograph (not necessarily connected) of order $n_2<n_1$. Then $M_{G_1} > M_{G_2}$.
\end{corollary}
\begin{proof}
The result is easily verified for $n_1=2$, so we may assume that $n_1\geq 3.$  From  Theorem \ref{SecondMainTheorem} and Corollary \ref{DisconnectedUpperBound}, we know that $M_{G_1} \geq M_{S_{n_1}} > \tfrac{n_1}{2},$ and $M_{G_2}\leq \tfrac{n_2+1}{2}$, respectively.  Since $n_1>n_2$, it follows that $M_{G_1}>\tfrac{n_1}{2}\geq \tfrac{n_2+1}{2}>M_{G_2}.$
\end{proof}

\section{Disconnected cographs}\label{DisconnectedSection}

In this section we investigate extremal structures for the mean order of connected induced subgraphs in disconnected cographs.
For any order $n $, the edgeless cograph   $\overline{K_{n}}$ has minimum mean. We show that the maximum mean among all disconnected cographs of order $n\geq 2$ is obtained by the disjoint union of $K_1$ and the cograph of maximum mean among all connected cographs of order $n-1$.  In particular, for $n\geq 8$, the cograph $K_1\cup K_{1,n-2}$ has maximum mean among all disconnected cographs of order $n$.

We will refer to the expressions given below for various means, which are easily verified.
\begin{lemma}\label{Means}
The following formulae hold.
\begin{enumerate}
    \item For $n \ge 3$, $M_{K_{1} \cup K_{1,n-2}} = \frac{(n-1)+n2^{n-3}}{(n-1)+2^{n-2}}=\frac{n}{2} + \frac{\frac{3n}{2}-\frac{n^{2}}{2}-1}{2^{n-2}+(n-1)}.$
    \item For $n \ge 4$, $M^{\ast}_{K_{1,n-3}} = \frac{(n-1)2^{n-4}-1}{2^{n-3}-1} =\frac{n}{2} + \frac{\frac{n}{2}-2^{n-4}-1}{2^{n-3}-1}.$
    \item For $n \ge 4$, $M_{K_{2,n-3}}= \frac{3n2^{n-4}-2^{n-4}+n-5}{3 \cdot 2^{n-3}+n-4}=\frac{n}{2} + \frac{-2^{n-4}+3n-5-\frac{n^{2}}{2}}{n + 3 \cdot 2^{n-3} - 4}.$
    \item For $n \ge 4$, $M_{K_{1,n-3}} = \frac{(n-3)+(n-1)2^{n-4}}{(n-3)+2^{n-3}}=\frac{n}{2} + \frac{\frac{5n}{2} - \frac{n^{2}}{2}-3-2^{n-4}}{2^{n-3}+n-3}.$
\end{enumerate}
\end{lemma}

In order to establish our main result we begin by proving some useful inequalities between the means described in Lemma \ref{Means}.

\begin{lemma}\label{MeanInequalities}
The following inequalities hold.
\begin{enumerate}
\item $M^{\ast}_{K_{1,n-3}} < M_{K_{1} \cup K_{1,n-2}}$ for $n\geq 8$.
\item $M_{K_{2,n-3}} < M_{K_{1} \cup K_{1,n-2}}$ for $n\geq 9$.
\item $M_{K_{1,n-3}} < M_{K_{1} \cup K_{1,n-2}}$ for $n \geq 4$.
\end{enumerate}
\end{lemma}

\begin{proof}
(a) From Lemma \ref{Means},  $M^{\ast}_{K_{1,n-3}} < M_{K_{1} \cup K_{1,n-2}}$ if and only if
\begin{align*}
     \frac{\frac{3n}{2}-\frac{n^{2}}{2}-1}{2^{n-2}+(n-1)} - \frac{\frac{n}{2}-2^{n-4}-1}{2^{n-3}-1}
     > 0.
\end{align*}
The numerator of this expression, given by
\begin{align*}
     &\ \ \ \left( \frac{3n}{2} - \frac{n^{2}}{2} - 1 \right) \left(2^{n-3}-1 \right) - \left(\frac{n}{2} - 2^{n-4} -1
     \right) \left(2^{n-2}+n-1\right),
\end{align*}
simplifies to  $2^{n-4} \left[ 2^{n-2} -(n^2-2n-1) \right]$.  It follows readily by induction, that $2^{n-2} +2n +1 -n^{2} > 0$ for $n \geq 8$.

\noindent
(b) From Lemma \ref{Means},  $M_{K_{2,n-3}} < M_{K_{1} \cup K_{1,n-2}}$ if and only if
\[ \frac{\frac{3n}{2} - \frac{n^{2}}{2} - 1}{2^{n-2} + (n-1)} -
     \frac{-2^{n-4}+3n-5-\frac{n^{2}}{2}}{n + 3\cdot2^{n-3} - 4} > 0. \]
The numerator of this expression, given by
\[ \left(\frac{3n}{2} - \frac{n^{2}}{2} - 1 \right) \left(n+ 3 \cdot 2^{n-3} -4 \right) -
     \left(2^{n-2}+n-1\right) \left(-2^{n-4} +3n - 5 - \frac{n^{2}}{2} \right),\]
 simplifies to $2^{n-4} \left[2^{n-2}-2n-n^{2}+13+ \frac{n-1}{2^{n-4}} \right].$  It follows readily by induction on $n \ge 9$, that
$2^{n-2}-2n-n^{2}+13 > 0$.

\noindent
(c) From Lemma \ref{Means},  $M_{K_{1,n-3}} < M_{K_{1} \cup K_{1,n-2}}$ if and only if
\[  \frac{\frac{3n}{2} - \frac{n^{2}}{2} - 1}{2^{n-2} +(n-1)} -
     \frac{\frac{5n}{2} - \frac{n^{2}}{2}-3-2^{n-4}}{2^{n-3}+n-3} > 0. \]
The numerator of this expression, given by
\[ \left( \frac{3n}{2} - \frac{n^{2}}{2} -1 \right) \left(2^{n-3}+n-3 \right) -
     \left(2^{n-2}+n-1\right) \left(\frac{5n}{2} - \frac{n^{2}}{2} - 3 - 2^{n-4} \right), \]
     simplifies to $2^{n-4} \left[2^{n-2} + n^{2} +9 -6n \right].$  Once again it is readily shown by induction on $n \ge 4$, that $2^{n-2} + n^{2} + 9 -6n > 0$.
\end{proof}

Now we are ready to prove the main result of this section.

\begin{theorem}
For each $n\geq 1$, let $Q_n$ denote the cograph of maximum mean among all connected cographs of order $n$.  If $G$ is a disconnected cograph of order $n \geq 2$, then $M_{G} \leq M_{K_{1} \cup Q_{n-1}}$.
\end{theorem}
\begin{proof}
We refer the reader to Table \ref{MaxTable} and Theorem \ref{FirstMainTheorem} for an explicit description of $Q_n$ for all $n$.  Using a computer algebra system, we have verified the statement for $n\leq 8$.

We may now assume that $n\geq 9$.  Suppose $G$ is a disconnected cograph of order $n\geq 9$, and we wish to show that $M_G\leq M_{K_1\cup K_{1,n-2}}$.  Let $H_{1},...,H_{k}$ denote the components of $G$. Then
\[\Phi_{G}(x)= \Phi_{H_{1}}(x) + ... + \Phi_{H_{k}}(x) \]
Hence $M_{G}$ can be expressed as a convex combination as follows:
\[M_{G} = c_{1}M_{H_{1}} + ... + c_{k}M_{H_{k}}. \]
Let $M^{'} = \max\{M_{H_{i}}: i \in \{1,...,k\}\}$. Then $M_{G} \leq M^{'}$. Let $H_i$ be such that $M_{H_{i}} = M^{'}$.
By Corollary \ref{ComparingOrder}, $H_{i}$ is a component of $G$ of largest order.

\noindent{\bf Case 1:} Suppose $G$ does not have a component of order $n-1$. Then the largest possible order of a component of $G$ is $n-2$, and by Corollary \ref{FirstMainTheorem}, it follows that $M_{G} < M_{K_{1,n-3}}$. By Lemma \ref{MeanInequalities}(c),
$M_{K_{1,n-3}} < M_{K_{1} \cup K_{1,n-2}}.$ Hence $M_{G} < M_{K_{1} \cup K_{1,n-2}}$.

\noindent{\bf Case 2:} Suppose $G$ has a component of order $n-1$. It follows that $G=K_{1} \cup H_{2}$,
where $H_{2}$ is a connected cograph of order $n-1$. In this case
\begin{align*}
     \Phi_{G}(x)=x + \Phi_{H_{2}}(x)
\end{align*}
{\bf Subcase 1:} Suppose $H_{2}$ has a universal vertex $v$. Then $K_{1,n-2}$ is a spanning subgraph of $H_{2}$. Thus
\begin{align*}
     \Phi_{G}(x) &= x + \Phi_{K_{1,n-2}}(x) + \Phi_{H_{2}-v}(x) - (n-2)x \\
     &= \Phi_{K_{1} \cup K_{1,n-2}}(x) + \Phi_{H_{2}-v}(x) - (n-2)x.
\end{align*}
Since $H_{2}-v$ is a cograph of order $n-2$, it follows that
\begin{align*}
     M_{G} = c_{1}M_{K_{1} \cup K_{1,n-2}} + c_{2}M^{\ast}_{H_{2}-v},
\end{align*}
and by Lemma \ref{LogDeriv}, $M^{\ast}_{H_{2}-v} < M^{\ast}_{K_{1,n-3}}$. By Lemma
\ref{MeanInequalities}(a), $M^{\ast}_{K_{1,n-3}} < M_{K_{1} \cup K_{1, n-2}}$. Hence
\begin{align*}
     M_{G} < M_{K_{1} \cup K_{1,n-2}}.
\end{align*}
{\bf Subcase 2:} Suppose $H_{2}$ does not have a universal vertex. Since $H_{2}$ is connected,
$H_{2}=G_{1}+G_{2}$, where $G_{1}$ and $G_{2}$ are cographs of orders $s$ and $n-s-1$ respectively, where $2 \leq s \leq \lfloor \frac{n}{2} \rfloor$. Thus
\begin{align*}
     \Phi_{G}(x) = (x + \Phi_{K_{s,n-s-1}}(x)) + (\Phi_{G_{1}}(x) - sx) + (\Phi_{G_{2}}(x) - (n-s-1)x).
\end{align*}
Hence
\begin{align*}
     M_{G} = c_{1}M_{K_{1} \cup K_{s,n-s-1}} + c_{2}M^{\ast}_{G_{1}} +c_{3}M^{\ast}_{G_{2}}.
\end{align*}
Since $G_{1}$ and $G_{2}$ are cographs of order at most $n-3$, it follows from Lemmas \ref{IncreaseLemma} and \ref{LogDeriv} that
$M^{\ast}_{G_{1}}, M^{\ast}_{G_{2}} < M^{\ast}_{K_{1,n-4}} < M^{\ast}_{K_{1,n-3}}$. By Lemma \ref{MeanInequalities}(a), $M^{\ast}_{K_{1,n-3}} < M_{K_{1} \cup K_{1,n-2}}$. Further, $M_{K_{1} \cup K_{s,n-s-1}}$ can itself be expressed as a convex combination of $M_{K_{1}}$ and $M_{K_{s,n-s-1}}$. From this and Theorem \ref{FirstMainLemma}, it follows that
$M_{K_{1} \cup K_{s,n-s-1}} < M_{K_{2,n-3}}$. By Lemma \ref{MeanInequalities}(b), $M_{K_{2,n-3}} < M_{K_{1} \cup K_{1,n-2}}$ for $n \geq 9$.
Hence $M_{G} < M_{K_{1} \cup K_{1,n-2}}$.
\end{proof}

\section{Conclusion}
In this article we studied the extremal structures for the mean order of the connected induced subgraphs of cographs, in both the connected and disconnected cases. For a connected cograph $G$ of order $n\geq 1$, we proved that
\[
\tfrac{n}{2}< M_G\leq \tfrac{n+1}{2},
\]
with equality if and only if $n=1.$  This means that the \emph{density} of $G$, defined as $\tfrac{M_G}{n}$, is always close to $\tfrac{1}{2}$; to be precise, the density of any infinite sequence of distinct cographs tends to $\tfrac{1}{2}.$  This contrasts the situation for trees, where the denisty lies between $\tfrac{1}{3}$ and $1$, and these bounds are asymptotically sharp (see \cite{j1} for details).

In Section \ref{LocalSection}, we showed that for connected cographs, the local mean at each vertex is at least as large as the global mean, as was the case for trees. We demonstrated that this property does not extend to all connected graphs, providing an infinite family of counterexamples.  It remains an interesting open problem to characterize those graphs for which the local mean at each vertex is at least as large as the global mean.

Finally, it remains an interesting open problem to determine the extremal structures for the mean connected induced subgraph order among all connected graphs of order $n$.  We conjecture that the minimum mean connected induced subgraph order is obtained by the path $P_n$, as for trees, and we have verified this statement for $n\leq 9$.  Determining the structure of graphs with maximum mean seems much more difficult.  Using a computer algebra system, we have determined that the maximum is not obtained by a tree for $3\leq n\leq 9$, but instead by a $2$-connected graph.  The connected graphs with maximum mean are given in Table \ref{MaxTable2}.  In the table, $\Theta_{i,j,k}$ denotes the \emph{$(i,j,k)$-theta graph}, obtained from the graph on two vertices with three multiedges between them by subdividing the first edge $i$ times, the second edge $j$ times, and the third edge $k$ times, and $G\square H$ denotes the \emph{Cartesian product} of graphs $G$ and $H$.

\begin{table}[h]
\begin{center}
\begin{tabular}{c c c}
Order & Graph & Global Mean\\
$3$ & $K_3$ & $12/7$\\
$4$ & $K_{2,2}$ & $28/13$\\
$5$ & $\Theta_{1,1,1}$ & $69/26$\\
$6$ & $\Theta_{2,1,1}$ & $67/21$\\
$7$ & $\Theta_{2,2,1}$ & $83/22$\\
$8$ & $\Theta_{2,2,2}$ & $22/5$\\
$9$ & $P_3\square P_3$ & $996/197$ 
\end{tabular}
\caption{The graphs of maximum global mean among all connected graphs of small order.\label{MaxTable2}}
\end{center}
\end{table}



\providecommand{\bysame}{\leavevmode\hbox to3em{\hrulefill}\thinspace}
\providecommand{\MR}{\relax\ifhmode\unskip\space\fi MR }
\providecommand{\MRhref}[2]{%
  \href{http://www.ams.org/mathscinet-getitem?mr=#1}{#2}
}
\providecommand{\href}[2]{#2}

\end{document}